\documentclass{amsart}
\usepackage{amssymb}

\usepackage[british,UKenglish,USenglish,american]{babel}
\usepackage{graphicx}
\usepackage{fancyhdr}
\usepackage{rotating}
\usepackage{amsmath}
\usepackage{amssymb}
\usepackage{amsthm}
\usepackage{tikz}
\usepackage{url}
\usepackage{enumerate}
\usepackage{mathtools}
\usepackage{setspace}
\usepackage{color}
\usepackage[normalem]{ulem}
\usetikzlibrary{positioning}

\newcommand{\Aut}{\operatorname{Aut}}

\newcommand{\Soc}{\operatorname{Soc}}

  \newtheorem*{notation*}{Notation}

\newcommand{\du}{\overline{U}}

\newtheorem{theorem}{Theorem}[section]
\newtheorem{lemma}[theorem]{Lemma}
\newtheorem{fact}[theorem]{Fact}
\newtheorem*{theorem*}{Theorem}
\newtheorem*{definition*}{Definition}
\newtheorem*{conjecture*}{Conjecture}
\newtheorem*{maintheorem*}{Main Theorem}
\newtheorem*{lemma*}{Lemma}

\newtheorem{propo}[theorem]{Proposition}
\theoremstyle{definition}
\newtheorem{conjecture}[theorem]{Conjecture}
\newtheorem{definition}[theorem]{Definition}
\newtheorem{remark}[theorem]{Remark}
\newtheorem*{remark*}{Remark}

\newtheorem*{claim*}{Claim}

\newtheorem{example}{Example}

\title[Fixed point subgroups of a supertight automorphism]{Fixed point subgroups of a supertight automorphism}

\author[U. Karhum\"{a}ki]{Ulla Karhum\"{a}ki}
\address{University of Helsinki}
\email{ukarhumaki@gmail.com}

\thanks{The author is funded by the Finnish Science Academy grant no: 338334.}

\begin{document}
\maketitle

\maketitle

\begin{abstract} Let $G$ be an infinite simple group of finite Morley rank and $\alpha$ a supertight automorphism of $G$ so that the fixed point subgroup $P_n:=C_G(\alpha^n)$ is pseudofinite for all $n\in \mathbb{N}\setminus\{0\}$. It is know (using CFSG) that the socle $S_n:={\rm Soc}(P_n)$ is a (twisted) Chevalley group over a pseudofinite field. We prove that there is $r\in \mathbb{N}\setminus\{0\}$ so that for each $n$ we have $[P_n:S_n] < r$ and that there is no $m \in \mathbb{N}\setminus \{0\}$ so that for each $n$ the sizes of the Sylow $2$-subgroups of $S_n$ are bounded by $m$. We also note that in the recent identification result of $G$ under the assumption ${\rm pr}_2(G)=1$, the use of CFSG is not needed. \end{abstract}

\section{Introduction}Groups of finite Morley rank are groups equipped with a notion of dimension which assigns to every definable set a dimension, called the \emph{Morley rank}, satisfying well-known axioms given in \cite{Borovik-Nesin, ABC}. It is conjectured independently by Cherlin and Zilber that simple groups of finite Morley rank are isomorphic to Chevalley groups over algebraically closed fields \cite{Cherlin1979, Zilber1977}.

In \cite{Ugurlu2009}, U\u{g}urlu defined a \emph{supertight automorphism} for groups of finite Morley rank which behaves like a very large power of the Frobenius map. Consequently, there is an ongoing project, which we call the \emph{supertight programme}, towards proving that the Cherlin-Zilber conjecture is equivalent to another conjecture, which proposes that the fixed point subgroup of a generic automorphism of an infinite simple group of finite Morley rank is pseudofinite \cite{Karhumaki-Ugurlu, Deloro-Karhumaki-Ugurlu}. This project and its aims are discussed in \cite[Introduction]{Karhumaki-Ugurlu}. 

The definition of the \emph{definable closure} $\overline{C_{H}(\alpha)}$ is given in Section~\ref{subsec:fRM}.

\begin{definition*}An automorphism $\alpha$ of a connected group of finite Morley rank $G$ is called \emph{tight} if, for any connected definable and set-wise $\alpha$-invariant subgroup $H$ of $G$, we have $\overline{C_{H}(\alpha)}=H$. Further, $\alpha$ is called \emph{supertight} if the following holds. \begin{enumerate}
\item For each $n\in \mathbb{N} \setminus \{0\}$ the power $\alpha^n$ is a tight automorphism of $G$.
\item \label{fixed-points} For any $m,n\in \mathbb{N} \setminus \{0\}$, if $m|n$ and $m\neq n$, then $C_G(\alpha^m)<  C_G(\alpha^n)$.
\end{enumerate} 
\end{definition*}

Let $G$ be a simple group of finite Morley rank with a supertight automorphism $\alpha$ whose fixed point subgroup $C_G(\alpha^n)$ is pseudofinite for each $n\in \mathbb{N}\setminus \{0\}$. Set $S:={\rm Soc}(C_G(\alpha))$ to be the socle of $C_G(\alpha)$. It is known that $S$ is isomorphic to a (twisted) Chevalley group $X(F)$ over a pseudofinite field $F$. This fact uses CFSG. By the definition of $\alpha$, the simple group $G$ is equal to the definable closure $\overline{S}$ and therefore it is expected that $G$ is isomorphic to a Chevalley group $X(K)$, of the same Lie type $X$ as $S$, over an algebraically closed field $K$. If the Pr\"{u}fer $2$-rank ${\rm pr}_2(G)$ (defined in Section~\ref{sec:CFSG}) of $G$ is $1$, then this expectation holds, that is, $G\cong {\rm PGL}_2(K)$ for some algebraically closed field $K$ of characteristic $\neq 2$ \cite[Theorem 1.2]{Karhumaki-Ugurlu}; important parts of this proof are to understand the structure of a centraliser of an involution in $C_G(\alpha)$ and to show that the quotient $C_G(\alpha)/S$ is finite. Also, as the fact that $S\cong X(F)$ is used, CFSG is used in the background of the proof.

The main results in this paper generalise some of the work in \cite{Karhumaki-Ugurlu}:

\begin{theorem}\label{th:sylows}Let $G$ be a simple group of finite Morley rank with a supertight automorphism $\alpha$ whose fixed point subgroup $P_n:=C_G(\alpha^n)$ is pseudofinite for all $n\in \mathbb{N}\setminus \{0\}$. Set $S_n:={\rm Soc}(P_n)$. Then the following holds. \begin{enumerate}[(a)]
\item There is $r\in \mathbb{N}\setminus \{0\}$ so that for each $n$, $[P_n:S_n] < r$.
\item There is no $m \in \mathbb{N}\setminus \{0\}$ so that for each $n$ the sizes of the Sylow $2$-subgroups of $S_n$ are at most $m$.
\end{enumerate}
\end{theorem}

\begin{propo}\label{th:CFSG}Theorem 1.2 in \cite{Karhumaki-Ugurlu} does not use CFSG.
\end{propo}

Our results are almost surely needed in the future when aiming to prove the expectation that $G$ is isomorphic to $X(K)$. Actually, we believe that in the future work one needs a sligthly more general set-up than in Theorem~\ref{th:sylows} and therefore we prove results in this level of generality (Proposition~\ref{propo:finite-index} and Proposition~\ref{lemma:inf-sylow}).

Our paper is organised as follows. In Section~\ref{sec:preliminaries} we give all the background results which are needed in the proof of Theorem~\ref{th:sylows} and in Section~\ref{sec:proofs} we prove Theorem~\ref{th:sylows}. Then, in Section~\ref{sec:CFSG}, we discuss the role of CFSG  in the supertight programme.

\section{Preliminaries}\label{sec:preliminaries}

 \subsection{Chevalley groups}\label{subsec:Chevalley}We denote a Chevalley group of \emph{Lie type} $X$ over an arbitrary field $k$ by $X(k)$, where $X$ comes from the list $A_n $, $B_n$, $C_n$, $D_n$, $E_6$, $E_7$, $E_8$, $F_4$, $G_2$; the subscript is called the \emph{Lie rank} of $X(k)$. If the Dynkin diagram has a non-trivial symmetry and the field $k$ satisfies suitable additional conditions, then $X(k)$ may be of \emph{twisted type}. In general, we refer the reader unfamiliar with (twisted) Chevalley groups to \cite{Carter1971}. Below, we discuss those properties of simple (twisted) Chevalley groups which are used in this paper, for all that we cite \cite{Carter1971}.

Let $G=X(k)$ be a simple (twisted) Chevalley group over a finite (or a pseudofinite) field $k$. Then $G$ is generated by the\emph{ root subgroups} $X_s$, each of which is isomorphic to the additive group $k^+$ (here $s\in \Phi$ is a root). Denote by $U$ the subgroups of $G$ generated by $X_s$ for all positive roots $s \in \Phi^+$. There is a surjective homomorphism $\phi_s$ from ${\rm SL}_2(k)$ to the group $R_s=\langle X_s, X_{-s}\rangle \leqslant G$ with ${\rm Ker}(\phi_s)\leqslant \{\pm I_2\}$. So either $R_s \cong {\rm SL}_2(k)$ or $R_s \cong {\rm PSL}_2(k)$. We call the group $R_s$ a \emph{${\rm SL}_2$-rootsubgroup} of $G$.

Let $0\neq \lambda \in k$. Define $h_s(\lambda)\in G$ and $n_s \in G$ respectively as the images of the diagonal matrix $\begin{pmatrix} \lambda & 0 \\ 0 &\lambda^{-1}  \end{pmatrix}$ and the monomial matrix  $\begin{pmatrix} 0 & 1 \\ -1 & 0 \end{pmatrix}$  under the map $\phi_s$. Then define $T=C_G(T)=\langle h_s(\lambda): s\in \Phi \rangle$, $N_G(T)=\langle T, n_s : s\in \Phi \rangle$ and $B=U \rtimes T.$ We have $N_G(U)=B$ and $N_B(T)=T$. The groups $B$, $U$ and $T$ are respectively called the \emph{Borel subgroup}, the \emph{unipotent subgroup} and the \emph{maximal split torus} of $G$. The unipotent subgroup $U$ is the Fitting subgroup of $B$ (i.e. the maximal normal nilpotent subgroup).

There are four types of generators of automorphisms of $G=X(k)$; they generate \emph{inner}, \emph{diagonal}, \emph{field} and \emph{graph} automorphisms. We denote by ${\rm Aut}(X)$, ${\rm Inn}(X)$, ${\rm Diag}(X)$, ${\rm Aut}(k)$ and ${\rm Grp}(X)$ the group of all, inner, diagonal, field and graph automorphisms of $X(k)$, respectively. Any element of ${\rm Aut}(X)$ factors uniquely into the product $xdgf$, where $x\in G$, $d\in {\rm Diag}(X)$, $g \in {\rm Grp}(X)$ and $f\in {\rm Aut}(k)$.

\subsection{Pseudofinite groups}Fix a language $\mathcal{L}$. A \emph{definable set} in an
$\mathcal{L}$-structure $\mathcal{M}$ is a subset $X\subseteq M^n$ which is the set of realisations of a first-order $\mathcal{L}$-formula $\varphi$. In this paper,  definable means definable possibly with parameters and we consider groups (resp. fields) as structures in the pure group (resp. field) language $\mathcal{L}_{gr}$. An \emph{$\mathcal{L}$-sentence} is an $\mathcal{L}$-formula in which all variables are bound by a quantifier. Two $\mathcal{L}$-structures $\mathcal{M}$ and $\mathcal{N}$ are \emph{elementarily equivalent}, denoted by $\mathcal{M} \equiv \mathcal{N}$, if they satisfy the same $\mathcal{L}$-sentences.

\begin{definition}A \emph{pseudofinite} group is an infinite group which satisfies every first-order sentence of $\mathcal{L}_{gr}$ that is true of all finite groups. \end{definition}

Let $\{ \mathcal{M}_i : i \in I \}$ be a family of
$\mathcal{L}$-structures and $\mathcal{U}$ be a non-principal
ultrafilter on $I$. We denote the ultraproduct, with respect to the ultrafilter $\mathcal{U}$, of the $\mathcal{L}$-structures $\mathcal{M}_i$ by $ \prod_{i \in I} \mathcal{M}_i / \mathcal{U}$. By \L o\'{s}'s Theorem, an infinite group (resp. $\mathcal{L}$-structure) is pseudofinite if and only if it is elementarily equivalent to a non-principal ultraproduct of finite groups (resp. $\mathcal{L}$-structures) of increasing orders, see \cite{Macpherson2018}.

Let $G$ be a group and $k$ be a positive integer. We say that $G$ is of \emph{centraliser dimension} $k$ if the longest proper descending chain of centralisers in $G$ has length $k$. If such an integer $k$ exists, then $G$ is said to be of \emph{finite centraliser dimension}.

The \emph{socle} of a group $G$, denoted by ${\rm Soc}(G)$, is the subgroup generated by all minimal normal (not necessarily proper) subgroups of $G$. 

A subgroup $H$ of a group $G$ is \emph{subnormal} if there is a finite ascending chain of subgroups starting from $H$ and ending at $G$, so that each is a normal subgroup of its successor. This is denoted by $H \unlhd\unlhd G$.

The following result describes the pseudofinite fixed point subgroups of supertight automorphisms; it is an important part of the supertight programme and its proof uses CFSG.

\begin{fact}[{\cite[Propostition 3.1]{Karhumaki-Ugurlu}}] \label{propo:pf}Let $G$ be a pseudofinite group of finite centraliser dimension. If for any non-trivial $H  \unlhd\unlhd  G$ we have $C_G(H)=1$, then the following hold. \begin{enumerate}[(a)]
\item \label{socleG} ${\rm Soc}(G)$ is $\emptyset$-definable in $G$ and ${\rm Soc}(G) \cong X(F)$, where $X(F)$ is a (twisted) Chevalley group over a pseudofinite field $F$.
\item \label{abelian-by-finite} $G/{\rm Soc}(G)$ is abelian-by-finite.
\end{enumerate}
\end{fact}

To our knowledge, the following remark does not appear anywhere in the literature.

\begin{remark} In Fact~\ref{propo:pf} one cannot prove further that $G/{\rm Soc}(G)$ is finite (cf. \cite{Borovik-Karhumaki}). Consider a pseudofinite field $F\equiv \prod_{p_i\in P} \mathbb{F}_{p_i^{p_i}}/\mathcal{U}$, where $p_i$'s are primes growing without a bound. Then for all  $1\neq a_i\in {\rm Aut}(\mathbb{F}_{p_i^{p_i}})$ we have ${\rm Fix}_{\mathbb{F}_{p_i^{p_i}}}(a_i)=\mathbb{F}_{p_i}$. Let$$G\equiv \prod_{p_i\in P}({\rm PSL}_2(\mathbb{F}_{p_i^{p_i}}) \rtimes {\rm Aut}(\mathbb{F}_{p_i^{p_i}}))/\mathcal{U} \equiv {\rm PSL}_2(F) \rtimes \prod_{p_i\in P}{\rm Aut}(\mathbb{F}_{p_i^{p_i}})/\mathcal{U}.$$ Then $G$ satisfies the assumptions of Fact~\ref{propo:pf} and  $G/{\rm Soc}(G)\equiv \prod_{i\in I}{\rm Aut}(\mathbb{F}_{p_i^{p_i}})/\mathcal{U}$ is infinite.\end{remark}

\subsection{Groups of finite Morley rank}\label{subsec:fRM}Groups of finite Morley rank are groups equipped with a notion of dimension which assigns to every definable set $X$ a dimension, called the \emph{Morley rank} and denoted by ${\rm rk}(X)$, satisfying well-known axioms given in \cite{Borovik-Nesin, ABC}. The following basic properties of groups of finite Morley rank are given in \cite{Borovik-Nesin, ABC}: Let $G$ be a group of finite Morley rank. By the Baldwin-Saxl chain condition, the length in $G$ of any proper chain of centralisers is bounded and one may define the \emph{connected component} $H^\circ$ of any subgroup $H \leqslant G$ and the \emph{definable closure} $\overline{X}$ of any subset $X \subseteq G$ as follows. If $L \leqslant G$ is definable then $L^\circ$ is the intersection of definable subgroups of finite indices in $L$, $\overline{X}$ is the intersection of all definable subgroups of $G$ containing $X$, and, for any $H \leqslant G$, $H^\circ=H \cap \overline{H}^\circ$. Note that $\overline{H}^\circ=\overline{H^\circ}$. We will use the following fact repeatedly in our proofs.

\begin{fact}[{\cite[Section 5.5]{Borovik-Nesin}}]\label{fact:def_closure}
Let $G$ be a group of finite Morley rank. Then the following hold.\begin{enumerate}[(1)]
	\item If $A \leqslant G$ normalises $X \subseteq G$, then $\overline{A}$ normalises $\overline{X}$.
	\item For any $X \subseteq G$, we have $C_G(X)= C_G(\overline{X})$.
		\item If $A \leqslant G$, then $\overline{A}'=\overline{A'}$.
	\item If $A \leqslant G$ is solvable (resp. nilpotent) of class $n$, then $\overline{A}$ is solvable (resp. nilpotent) of class $n$.
	\item Let $A \leqslant B \leqslant G$. If $[B:A] < \infty$ then $[\overline{B}:\overline{A}] < \infty$.
\end{enumerate}
\end{fact}

In a group of finite Morley rank the Sylow $2$-subgroups are conjugate and the connected component $P^\circ$ of a Sylow $2$-subgroup $P$ of $G$ is the central product $U \ast T$  where $U$ is a \emph{$2$-unipotent} group (a definable and nilpotent $2$-group of bounded exponent) and $T$ is a \emph{$2$-torus} (i.e. $T \cong (\mathbb{Z}_{2^\infty})^\ell$ for some $\ell \in \mathbb{N} \setminus \{0\}$, where $\mathbb{Z}_{2^\infty}= \{x \in \mathbb{C}^{\times} : x^{2^n}=1 \, {\rm for} \,{\rm some} \,  n \in \mathbb{N} \}$ is the \emph{Pr\"{u}fer $2$-group}) \cite[I.6]{ABC}. Groups of finite Morley rank can be split into four cases based on the structure of $P^\circ $:
\begin{enumerate}
\item \emph{Even type}: $P^\circ$ is a non-trivial $2$-unipotent group.
\item \emph{Odd type}:  $P^\circ$ is a non-trivial $2$-torus.
\item \emph{Mixed type}: $P^\circ=U \ast T$ with $U\neq 1$ and $T\neq 1$.
\item \emph{Degenerate type}: $P^\circ$ is trivial.
\end{enumerate}  If the ambient group $G$ is infinite and simple, then $G$ cannot be of mixed type \cite{ABC}. Also, in this case, either $P=1$ or $P^\circ$ is infinite \cite{Borovik-Burdges-Cherlin}.

\section{Fixed point subgroups of supertight automorphisms}\label{sec:proofs}

\subsection{Pseudofinite fixed point subgroups}\label{sec:not}Let $G$ be a connected odd type group of finite Morley rank and $\alpha$ a supertight automorphism of $G$. Assume that the following holds for all $n\in \mathbb{N}\setminus\{0\}$. \begin{enumerate}[(i)]
\item $P_n:=C_{G}(\alpha^n)$ is pseudofinite and for any non-trivial $H_n \unlhd\unlhd P_n$ we have $C_{P_n}(H_n)=1$.
\item Set $S_n:={\rm Soc}(P_n)$. Then $\overline{S}_n=G$. \end{enumerate}

By Fact~\ref{propo:pf}, $S_n$ is isomorphic to a simple (twisted) Chevalley group ${X}(F_n)$ where $F_n$ is a pseudofinite field of characteristic $\neq 2$ (as $G$ is of odd type). As $S_1\leqslant S_n$ for all $n$ (if necessary, see \cite[Remark 5.2]{Karhumaki-Ugurlu}), in our proofs, we may freely replace $\alpha$ with its power and we know that $S_n$'s are of the same Lie type. Hence, we often omit the subscripts and simply write $P$ and $S \cong X(F)$. 

Set $P\equiv \prod_{i\in I}P_i/\mathcal{U}$, where $P_i$'s are finite. By Fact~\ref{propo:pf}, we have $$S={\rm Soc}(P)\cong {X}(F) \equiv \prod_{i\in I} X(\mathbb{F}_{q_i})/\mathcal{U}=\prod_{i\in I}{\rm Soc}(P_i)/\mathcal{U}$$and hence $$P/S \equiv \prod_{i\in I} P_i /\Soc (P_i)/\mathcal{U} \hookrightarrow \prod_{i\in I} {\rm Out}(\Soc (P_i))/\mathcal{U},$$ where $q_i$ is odd, the Lie type $X$ is fixed (modulo $\mathcal{U}$) and ${\rm Out}(\Soc (P_i))$ is the group of outer automorphisms of $\Soc (P_i)$. Any element in ${\rm Out}(\Soc (P_i))$ factors uniquely into the product of a diagonal automorphism, a graph automorphism and a field automorphism, the finite groups ${\rm Diag}(X)$ and ${\rm Graph}(X)$ do not depend on $i$ and the group $\prod_{i\in I} \Aut(\mathbb{F}_{q_i})/\mathcal{U}$, being an ultraproduct of finite cyclic groups, has boundedly many elements of each finite order (and these bounds do not depend on $i$). It easily follows that $P/S$ has boundedly many elements of each finite order: Assume that $x_i \in {\rm Out}(\Soc (P_i))$ is of order $\ell$. We have $x_i=dgf_i$, where $d\in {\rm Diag}(X)$, $g\in {\rm Graph}(X)$ and $f_i \in \Aut(\mathbb{F}_{q_i})$. So it is enough to observe that there are only boundedly many choices for $f_i$. Let $T_i$ be the diagonal subgroup of $\Soc (P_i)$. Now $(dgf_i)^\ell$ and $d$ act trivially on $T_i$ and hence so does $(gf_i)^\ell$. But any graph automorphism either permutes non-trivially the ${\rm SL}_2$-rootsubgroups of $\Soc (P_i)$ corresponding to the (finite) Dynkin diagram of $\Soc (P_i)$ or is trivial (\cite{Carter1971}), and any field automorphism leaves the diagonal subgroups of the ${\rm SL}_2$-rootsubgroups of $\Soc (P_i)$ invariant. It easily follows that $(gf_i)^\ell$ acts on $T_i$ as $f_i^\ell$ so the field automorphism $f_i^\ell$ must be trivial, as we wanted. Finally note that we also have (again by Fact~\ref{propo:pf}) that $ P \leqslant {\rm Aut}(S)$. So any element of $P$ is a product $sdgf$, where $s\in S$, $d\in {\rm Diag}(X) $, $g\in {\rm Graph}(X)$ and $f\in {\rm Aut}(F)$.

\subsection{Killing the field automorphisms} In this section, we adobt the notation from Section~\ref{sec:not}. Using similar arguments as in \cite[Section 5.3]{Karhumaki-Ugurlu}, we prove that the fixed point subgroup $P_n=C_G(\alpha^n)$ is a finite extension of its socle $S_n={\rm Soc}(P_n)$. 

In what follows, we denote by $T_n$ and $U_n$ a maximal split torus and a unipotent subgroup of some fixed Borel subgroup $B_n$ of $S_n$ (Subsection~\ref{subsec:Chevalley}). Also, set $\mathcal{D}={\rm Diag}(X)$ and $\mathcal{G}={\rm Graph}(X)$. Note that the finite groups $\mathcal{D}$ and $\mathcal{G}$ do not depend on $n$ and that both $T_n$ and $U_n$ are set-wise invariant under the action of $\mathcal{D}$, $\mathcal{G}$ and ${\rm Aut}(F_n)$.

\begin{lemma}\label{lemma:intersections-P}There is $\ell \in \mathbb{N} \setminus\{0\}$ so that for each $n$ we have $[C_{P_n}(T_n):T_n]< \ell$. Also, $[\overline{T}_n\cap P_n:T_n]$ and $\overline{U}_n\cap C_{P_n}(T_n)$ are finite.
\end{lemma}
\begin{proof} Let $x_n\in C_{P_n}(T_n)$. Then $x_n=s_ndga_n$, where $s_n\in N_{S_n}(T_n)$, $d\in \mathcal{D}$, $g\in \mathcal{G}$ and $a_n\in {\rm Aut}(F_n)$. Since $S_n$'s are of the same Lie type, there is $k$ so that, for each $n$, we have $[N_{S_n}(T_n):T_n]=k$. So to prove that there is $\ell$ so that $[C_{P_n}(T_n):T_n]< \ell$, it is enough to show that there are only boundedly many choices for $a_n$. As $d$ centralises $T_n$ so does $s_nga_n$. So $s_nga_n$ centralises each $1$-dimensional subtorus (i.e. a diagonal subgroup of a ${\rm SL}_2$-rootsubgroup of $S_n$) of $T_n$; let $J_n$ be one such subtorus. Then either $J_n\cong F_n^\times$ or $J_n\cong (F_n^\times)^2$ (the set of squares). It easily follows (by finiteness of $\mathcal{G}$) that there is $r \in \mathbb{N}\setminus\{0\}$ (which does not depend on $n$) and $s'_n\in S_n$ so that $(s_nga_n)^r$ acts on $J_n$ as the product $s_n'a_n^r$ of an inner automorphism $s_n'$ and a field automorphism $a_n^r$. Hence $s'_n \in S_n$ normalises $J_n$, thus either centralises or inverts it. So $a_n^r$ acts either trivially or by inversion on the field $F_n$; the latter is not possible so we get $a_n^r=1$. Now, by Section~\ref{sec:not}, there are only boundedly many choices for $a_n$.

Recall that $C_{S_n}(T_n)=T_n$. Now, as $\overline{T}_n$ is abelian (Fact~\ref{fact:def_closure}), we have $\overline{T}_n \cap P_n  \leqslant C_{P_n}(T_n)$. By above, $[\overline{T}_n\cap P_n:T_n]$ is finite.

Finally, let $x\in \overline{U}_n\cap C_{P_n}(T_n)$. Then $x_n=s_ndga_n$ as above and $dga_n\in N_G(\du)$ (Fact~\ref{fact:def_closure}). So $\overline{U}_n^{s_n}=\overline{U}_n$ and hence $\langle U_n, U_n^{s_n} \rangle$ is a nilpotent (again by Fact~\ref{fact:def_closure}) subgroup of $S_n$; it easily follows that $\langle U_n, U_n^{s_n} \rangle$ is contained in some Borel subgroup of $S_n$ and hence in $B_n=U_n\rtimes T_n$. So $s_n\in N_{B_n}(T_n)=T_n$. Since $d$ also centralises $T_n$ so does $ga_n$; it easily follows (using similar arguments as above) that $a_n=1$. So $\overline{U}_n\cap C_{P_n}(T_n)$ is a finite extension of $\overline{U}_n\cap \overline{T}_n \cap S_n=U_n\cap T_n=1$. \end{proof}

\begin{propo}\label{propo:finite-index} There is $r\in \mathbb{N}\setminus \{0\}$ so that for each $n$, $[P_n:S_n] < r$.
\end{propo}
\begin{proof}Assume first that for each $n$ there is $r_n$ so that $[P_n:S_n] <r_n$. We  prove that this implies that there is $r\in \mathbb{N} \setminus\{0\}$ so that for each $n$ we have $[P_n:S_n] < r$. We know that $P_n/S_n \cong N_{P_n}(T_n)/N_{S_n}(T_n)$ and that $[N_{S_n}(T_n):T_n]$ is bounded independently from $n$. By Lemma~\ref{lemma:intersections-P} it is enough to prove that there is $k\in \mathbb{N}\setminus \{0\}$ so that $[N_{P_n}(T_n):C_{P_n}(T_n)] < k$. Since $G$ is of finite centraliser dimension, by abelianity of $T_n$'s, there is $m\in \mathbb{N}\setminus \{0\}$ so that $C_G(T_m)=C_G(T_\ell)$ for any $\ell \geqslant m$. So after replacing $\alpha$ with a suitable power, we have $C_G(T_n)=C_G(T_1)$ for each $n$. Now $N_{P_n}(T_n)\leqslant N_G(\overline{T}_n)$ normalises $C_G(T_1)=C_G(\overline{T}_n)$ (Fact~\ref{fact:def_closure}) and we have $$N_G(\overline{T}_1)/C_G(T_1)\geqslant  N_{P_n}(T_n) C_G(T_1)/C_G(T_1) \cong N_{P_n}(T_n)/ C_{P_n}(T_n).$$ Since $N_{P_n}(T_n)/N_{S_n}(T_n)$ is finite (by assumption) and $\overline{T}_n\cap S_n =T_n$, we get $N^\circ_G(\overline{T}_n)=\overline{N^\circ_{S_n}(\overline{T}_n)}=\overline{N^\circ_{S_n}(T_n)}=\overline{T}_n^\circ$. So we may set $k=[N_G(\overline{T}_1):C_G(T_1)]$.

Now, we need to prove that $P_n/S_n$ is finite for each $n$. For the rest of the proof, we may omit the subscripts and simply write $P$ and $S\cong X(F)$. We have $B=U \rtimes T$ as above (a Borel subgroup of $S$). Set $O={\rm Out}(S) \cap P$. We need to prove that $O$ is finite. By Lemma~\ref{lemma:intersections-P}, it is enough to prove that all but finitely many elements of $O^\circ$ centralises $T$.

We first prove that $O^\circ$ centralises $T^\circ$. We know that $P/S\cong N_P(U)/N_S(U)=N_P(U)/B\cong N_P(T)/N_S(T)$ is abelian-by-finite (Fact~\ref{propo:pf}). So there is a finite extension $L$ of $N_S(T)$ (and thus of $T$) so that $N_P(T)'\leqslant L$. Taking connected components and definable closures, we get $\overline{N^\circ_P(T)}'\leqslant \overline{L}^\circ=\overline{N^\circ_S(T)}=\overline{T}^\circ$ (Fact~\ref{fact:def_closure}). Similarly, there is a finite extension $H$ of $B$ so that $N_P(U)'\leqslant H$ and $\overline{N^\circ_P(U)}' \unlhd \overline{B}^\circ$. Therefore $\overline{N^\circ_P(U)}$ is solvable (Fact~\ref{fact:def_closure}) and hence $\overline{N^\circ_P(U)}'$ is nilpotent by \cite{Nesin1989}. So the nilpotent normal subgroup $\overline{N^\circ_P(U)}'\cap S$ of $\overline{B}^\circ \cap S=B^\circ$ is contained in $U^\circ$ and thus equal to $(B^\circ)'=U^\circ$. The connected $\alpha$-invariant group $N=\overline{N^\circ_P(T)}'$ is either trivial or infinite. Now, if $\overline{N\cap P}=N$ is infinite, then so is $N\cap P$ and further, by Lemma~\ref{lemma:intersections-P}, so is $N\cap S\leqslant \overline{T}\cap S =T$. As $\overline{N^\circ_P(U)}=\overline{BO}^\circ$, we have $N \leqslant \overline{N^\circ_P(U)}'$. So, if $N$ is infinite, then $N \cap S  \leqslant T$ is an infinite subgroup of $\overline{N^\circ_P(U)}' \cap S = U^\circ$; a contradiction. Therefore $N$ is trivial and $O^\circ$ centralises $T^\circ$.

We are ready to conclude that $O$ is finite. Let $x\in O^\circ$. Then $x$ centralises $T^\circ$ and we have $x=dga$ where $d\in \mathcal{D}$, $g\in \mathcal{G}$ and $a\in {\rm Aut}(F)$. Let $J$ be a $1$-dimensional subtorus of $T$, so either $J\cong F^\times$ or $J\cong (F^\times)^2$. As $\mathcal{G}$ is finite, for all but finitely many choices of $a$, the automorphism $a$ acts trivially on $T^\circ$ (Section~\ref{sec:not}) and thus on $J^\circ=J\cap T^\circ$. We just need to observe that if $a' \in {\rm Aut}(F)$ acts trivially on $J^\circ$ then it acts trivially on $F^\times$, and is therefore trivial. Set $K:=\mathrm{Fix}_{F}(a')$. Since $[J^\circ:F^\times] < \infty$, we have $[K^\times:F^\times]=n$ and $F^\times=\bigcup_{j=1}^n x_jK^\times$ for some $x_j\in F$ and $n\geq 1$. If $n>1$, then there are distinct $s,t\in K^\times$ and $j\leq n$ so that $x_1+tx_2,x_1+sx_2\in
x_jK^\times$. But then the difference of $x_1+tx_2$ and $x_1+sx_2$ falls into two distinct cosets of $K^\times$ which cannot happen. So, by Section~\ref{sec:not}, there are only finitely many choices for $a$ and hence $O$ is finite. \end{proof}

\subsection{Growing the Sylow $2$-subgroups}We now briefly discuss the Sylow $2$-subgroups of the fixed point sugbroups $C_G(\alpha^n)$ when $\alpha$ and $G$ are concrete examples of a supertight automorphism and a simple group of finite Morley rank. The example of a supertight automorphism is the \emph{non-standard Frobenius automorphism}: Let $G$ be a simple Chevalley group, recognised as the group of $K$-rational points of an algebraically closed field $K :=\prod_{p_i \in P} \mathbb{F}^{alg}_{p_i}/\mathcal{U}$. Then the non-standard Frobenius automorphism $\phi_{\mathcal{U}}$ of $K$ induces on $G$ an automorphism. Here $P$ is an infinite set of prime numbers, $\mathcal{U}$ is a non-principal ultrafilter on $P$ and $\phi_{\mathcal{U}}$ is the map from $K$ to $K$ sending an element $[x_i]_{\mathcal{U}}$ to the element $[x_i^{p_i}]_{\mathcal{U}}$. It is clear that, for each $n\in \mathbb{N}\setminus \{0\}$, $\phi_{\mathcal{U}}^n$ is a supertight automorphism of $G$ with pseudofinite fixed point subgroup $G({\rm Fix}_{\phi_{\mathcal{U}}^n}(K))$. 

Below we prove that the Sylow $2$-subgroups of the fixed points of \emph{large powers} of the non-standard Frobenius automorphism $\phi_{\mathcal{U}}$ are big. However, it is worth noting that the fixed point subgroup of $\phi_{\mathcal{U}}$ may have very small Sylow $2$-subgroups:

\begin{example}Let $\Pi:=\{p \,\, {\rm is } \,\,  {\rm a}\,\,  {\rm prime }   : p \equiv 5({\rm mod}  \, 8)\}$, $K:=\prod_{p_i \in \Pi} \mathbb{F}^{alg}_{p_i}/\mathcal{U}$, $G:={\rm PGL}_2(K)$ and $\phi_{\mathcal{U}}$ the non-standard Frobenius automorphism of $K$ inducing on $G$ a supertight automorphism. Denote $F:={\rm Fix}_{\phi_{\mathcal{U}}}(K)$. Then $C_G(\phi_{\mathcal{U}})={\rm PGL}_2(F) \equiv \prod_{p_i\in \Pi} {\rm PGL}_2(\mathbb{F}_{p_i})/\mathcal{U}$ and, by the choice of $\Pi$, the Sylow $2$-subgroups of ${\rm PGL}_2(F)$ are Klein $4$-groups. \end{example}

Below we again adobt the notation from Section~\ref{sec:not}. 

\begin{propo}\label{lemma:inf-sylow} There is no $m \in \mathbb{N}\setminus \{0\}$ so that for each $n$ the sizes of the Sylow $2$-subgroups of $S_n$ are at most $m$. \end{propo}

\begin{proof} If $S_1$ (the socle of $P_1=C_G(\alpha)$) has an infinite Sylow $2$-subgroup, then we are done. So we may assume that $S_1$ has a finite Sylow $2$-subgroup, say $\Theta$. By Proposition~\ref{propo:finite-index} it is enough to prove that there is no $m \in \mathbb{N}\setminus \{0\}$ so that for each $n$ the sizes of the Sylow $2$-subgroups of the fixed point subgroups $P_n$ are at most $m$. Note that if the pseudofinite group $P_n$ has a finite Sylow $2$-subgroup then its Sylow $2$-subgroups are conjugate by \L os's Theorem (see e.g. \cite{Ugurlu2017}). So, towards a contradiction, assume that there is $m\in \mathbb{N} \setminus\{0\}$ so that for each $n$ any (equiv. some) Sylow $2$-subgroup of $P_n$ has size at most $m$. We may now omit subscripts and simply write $P$ and $S$.

We need the following result (which uses CFSG).

\begin{theorem}[Kondrat'ev \cite{Kondratev}]\label{th:normaliser-Sylow}Let $H=X(\mathbb{F}_q)$ be a finite simple group of Lie type with ${\rm char}(\mathbb{F}_q) \neq 2$ and let $T$ be a Sylow $2$-subgroup of $H$. If $N_H(T) \neq T$ then one of the following holds. \begin{enumerate}\item $|N_H(T)|=d$, where $d$ depends on the Lie type $X$ of $H$ only. \item $|N_H(T)/T|=d'$, where $d'$ depends on the Lie type $X$ of $H$ only. \item $N_H(T)=T \times C_1 \times  \cdots \times C_{m}$ where $m$ depends on the Lie rank of $H$ only and the orders of the cyclic groups $C_j$ depends on $q$ for all $j\in \{1, \ldots, m\}$.\end{enumerate} \end{theorem} 

Since the Sylow $2$-subgroup $\Theta$ of $S$ is finite and $S\equiv \prod_{i\in I} X(\mathbb{F}_{q_i}),$ Theorem~\ref{th:normaliser-Sylow} tells us that $N_{S}(\Theta)$ is either finite or infinite and abelian-by-finite. Moreover, as $\Theta$ is finite we have $N^\circ_G(\Theta)=C^\circ_G(\Theta)$ (\cite[Lemma 5.9]{Borovik-Nesin}) and, by Proposition~\ref{propo:finite-index}, $C^\circ_P(\Theta)=C^\circ_S(\Theta)$. So $\overline{N^\circ_{S}(\Theta)}=\overline{C^\circ_P(\Theta)}=C^\circ_G(\Theta)=N^\circ_G(\Theta)$ is abelian (possibly trivial) by Fact~\ref{fact:def_closure}.

Now, set $N^\circ=N^\circ_G(\Theta)$.  If $N^\circ=1$ then the finite and set-wise $\alpha$-invariant group $N$ is point-wise fixed by some power of $\alpha$; we may replace $\alpha$ by this power. Let $\Sigma$ be a Sylow $2$-subgroup of $G$ containing $\Theta$. Since $\Sigma$ satisfies the normaliser condition (\cite[Corollary 6.22]{Borovik-Nesin}), we have $\Sigma_1:=N_\Sigma(\Theta)> \Theta$. Note that $\Sigma_1$ is contained in a Sylow $2$-subgroup of $C_G(\alpha)$. Now $N^\circ_G(\Sigma_1)=C^\circ_G(\Sigma_1) \leqslant C^\circ_G(\Theta)= N^\circ$ and hence $N_{\Sigma}(\Sigma_1)$ is finite and point-wise fixed by some power of $\alpha$. After replacing $\alpha$ with a suitable power, $\Sigma_2 :=N_{\Sigma}(\Sigma_1)$ is contained in a Sylow $2$-subgroup of $C_G(\alpha)$ with $\Sigma_2> \Sigma_1$ (we apply the normaliser condition again). By repeating this procedure we get a Sylow $2$-subgroup of $C_G(\alpha)$ of size $> m$; a contradiction.

So we may assume that $N^\circ$ is infinite. Now, since $G$ is of odd type, the abelian group $N^\circ$ has a unique Sylow $2$-subgroup, which is either trivial or isomorphic to $(\mathbb{Z}_{2^\infty})^k$ for some $k\in \mathbb{N}\setminus \{0\}$ (\cite[Theorem 9.29]{Borovik-Nesin}). The former cannot happen: If $N^\circ$ has no $2$-elements then there is an involution $i$ in $ \Theta\cap( N \setminus N^\circ)$ and, by torsion-lifting (\cite[Section 5.5, Ex. 11]{Borovik-Nesin}), there is a $2$-element $in$ in the coset $ i N^\circ$. But then, as $n$ centralises $i$, $n$ is a $2$-element in the group $N^\circ$. So the unique Sylow $2$-subgroup of $N^\circ$, say $\Pi$, is isomorphic to $(\mathbb{Z}_{2^\infty})^k$. As $N^\circ$ is set-wise $\alpha$-invariant, so is $\Pi$. This means that for each $s\in \mathbb{N}\setminus\{0\}$, all elements of $\Pi$ of order $2^s$ are fixed by some power of $\alpha$ and therefore, after replacing $\alpha$ with a suitable power, $C_G(\alpha)$ has a Sylow $2$-subgroup of size $> m$. This contradiction proves the proposition.
\end{proof}


\subsection{Proof of the main theorem}We are ready to prove Theorem~\ref{th:sylows} as an easy corollary of Proposition~\ref{propo:finite-index} and Proposition~\ref{lemma:inf-sylow}.
\begin{proof}[Proof of Theorem~\ref{th:sylows}]Let $G$ be an infinite simple group of finite Morley rank with a supertight automorphism $\alpha$ whose fixed point subgroup $P_n:=C_G(\alpha^n)$ is pseudofinite for all $n\in \mathbb{N}\setminus \{0\}$. Let $1 < H_n$ be a subnormal subgroup of $P_n$, say $H_n = H_{n_1} \triangleleft \dots \triangleleft H_{n_m} = P_n$. Since $G$ is simple and $\alpha$ is tight, $\overline{H}_{n_m} = G$. Inductively, $\overline{H}_n = G$; in particular $C_G(H_n) = Z(G) = 1$, so $C_{P_n}(H_n) = 1$ and ${\rm Soc}(P_n)$ is isomorphic to a simple pseudofinite group $X(F_n)$ by Fact~\ref{propo:pf}. So $G$ contains involutions and is therefore either of even or of odd type.

If $G$ is of odd type then Theorem~\ref{th:sylows} holds by Proposition~\ref{propo:finite-index} and Proposition~\ref{lemma:inf-sylow}. If $G$ is of even type then it is isomorphic to a Chevalley group over an algebraically closed field $K$ of characteristic $2$ by \cite{ABC}. Theorem~\ref{th:sylows} easily follows (a rootsubgroup of ${\rm Soc}(P_1)$ is an infinite elementary abelian $2$-group and any field automorphism of ${\rm Soc}(P_n)$ in $P_n$ would extend to a field automorphism of $K$ in $G$).\end{proof}

\section{The role of CFSG in the supertight programme}\label{sec:CFSG}

We now discuss the role of the classification of finite simple groups (CFSG) in the supertight programme. Note first that one of the main goals of this programme is to prove the following (for more detail, see \cite[Introduction]{Karhumaki-Ugurlu}):

\begin{conjecture}\label{conj:tight}An infinite simple odd type group of finite Morley rank $G$ with a supertight automorphism $\alpha$ whose fixed point subgroup $C_G(\alpha^n)$ is pseudofinite for each $n\in \mathbb{N}\setminus \{0\}$ is isomorphic to a Chevalley group $X(K)$ over an algebraically closed field $K$ of characteristic $\neq 2$.\end{conjecture}

Let $G$ and $\alpha$ be as in Conjecture~\ref{conj:tight}. By Fact~\ref{propo:pf}, the socle $S_n={\rm Soc}(C_G(\alpha^n))$ is a (twisted) Chevalley group and hence has a \emph{${\rm BN}$-pair} (\cite{Carter1971}). The natural strategy for the identification of $G$ is to construct a ${\rm BN}$-pair for $G$ using the one of $S_n$ (for any fixed $n$). However, the proof that $S_n$ is a (twisted) Chevalley group heavily relies on Wilson's classification of simple pseudofinite groups (\cite{Wilson1995}) which in turn heavily relies on CFSG. It seems to the author that removing CFSG from Fact~\ref{propo:pf} in general is an extremely hard task (few more words about this is said in the very end of the paper). However, if $G$ is `small' then this can be easily done (and further, the subnormality assumption can be relaxed). This is explained below.

Consider a Chevalley group $H=X(K)$ over an algebraically closed field $K$. The maximal (finite) number of copies of the Pr\"{u}fer $2$-group $\mathbb{Z}_{2^\infty}$ in a Sylow $2$-subgroup of $H$ is called the \emph{Pr\"{u}fer $2$-rank} of $H$ and it is denoted by ${\rm pr}_2(H)$. One may measure the `size' of $H$ by its Pr\"{u}fer $2$-rank. The only simple Chevalley group over $K$ of ${\rm pr}_2(H)=1$ is ${\rm PSL}_2(K)={\rm PGL}_2(K)$.

Let then $H$ be a finite group. The \emph{rank} of $H$ is the smallest cardinality of its generating set and the \emph{$2$-rank} of $H$, denoted by $m_2(H)$, is the largest rank of its elementary abelian $2$-subgroup. The following result is a crucial part of CFSG, and does not itself use CFSG.

\begin{theorem}[See e.g. {\cite[page 6]{Gorenstein1983}}]\label{th:finite2-rank2} Let $H$ be a finite simple group with $m_2(H) \leqslant 2$. Then $H$ is isomorphic to one of the following groups $${\rm PSL}_2(q), {\rm PSL}_3(q), {\rm PSU}_3(q) \,\, {\rm for} \,\, {\rm odd} \,\, q, \,\, {\rm PSU}_3(4), A_7 \,\, {\rm or} \,\,  M_{11}.$$ \end{theorem}

The following easy observation proves that if ${\rm pr}_2(G)=1$ then the identification of $G$ (which is done in \cite{Karhumaki-Ugurlu}) does not require CFSG. 

\begin{propo}\label{propo:pf-noCFSG}Let $P$ be a pseudofinite group of finite centraliser dimension. Assume that $m_2(P) \leqslant 2$. If for any non-trivial $H  \unlhd P$ we have $C_P(H)=1$, then ${\rm Soc}(P)$ is isomorphic to one of the groups ${\rm PSL}_2(F), {\rm PSL}_3(F)$ or ${\rm PSU}_3(F)$ over a pseudofinite field $F$ of characteristic $\neq 2$. 

In particular, if the assumption `$m_2(P) \leqslant 2$' is replaced by an assumption `$P \leqslant G$, where $G$ is a connected odd type group of finite Morley rank with ${\rm pr}_2(G)=1$', then ${\rm Soc}(P) \cong {\rm PSL}_2(F)$ over a pseudofinite field $F$ of characteristic $\neq 2$.  \end{propo}

\begin{proof}Let $P\equiv \prod_{i\in I}P_i/\mathcal{U}$, where $P_i$ is a finite group. We argue modulo $\mathcal{U}$ even if implicitly. Note that $m_2(P_i) = m_2(P)$ (\L os's Theorem). By assumption $P$ has no abelian normal subgroups and therefore, by \L os's Theorem, neither does $P_i$. So the socle ${\rm Soc}(P_i)$ of the finite group $P_i$ is a finite direct product $M_{i,1}\times \cdots \times M_{i,\ell}$ of minimal normal groups. For each $m\in \{1, \ldots, \ell\}$, the minimal normal group $M_{i,m}$ is a finite direct product of isomorphic finite simple groups, each of which is by assumption of $2$-rank $2$ (a finite simple group, say $H$, of $m_2(H)\leqslant 2$ necessary has $m_2(H)=2$, see \cite[page 13]{Gorenstein1983}). So Theorem~\ref{th:finite2-rank2} implies that $\ell=1$ and ${\rm Soc}(P_i)$ is simple non-abelian with $m_2({\rm Soc}(P_i))=2$. Note also that there is no common bound on the size $|{\rm Soc}(P_i)|$ (when $i$ varies) as for otherwise there would be a bound on the size $|P_i| \leqslant |{\rm Aut}({\rm Soc}(P_i))|$, but $P$ is infinite. So, by Theorem~\ref{th:finite2-rank2}, ${\rm Soc}(P_i)$ is isomorphic to one of the groups ${\rm PSL}_2(q_i), {\rm PSL}_3(q_i)$ or ${\rm PSU}_3(q_i)$ for $q_i$ odd.

We still need to show that ${\rm Soc}(P)\equiv \prod_{i\in I} {\rm Soc}(P_i)/\mathcal{U}$; this can done exactly as in \cite[Proposition 3.1]{Karhumaki-Ugurlu}. We repeat the short argument here. As $q_i>8$, by \cite{Ellers-Gordeev1998}, there is $x_i \in {\rm Soc}(P_i)$ so that ${\rm Soc}(P_i)= x_i^{{\rm Soc}(P_i)} x_i^{{\rm Soc}(P_i)} =  x_i^{P_i} x_i^{P_i}.$ Moreover, by \cite{Ellers-Gordeev-Herzog}, there is $n$, which does not depend on $i$, so that the socle ${\rm Soc}(P_i)$ has the $n$-covering property (i.e. for any $s_i \in {\rm Soc}(P_i)$, $n$ products of the conjugacy class $s_i^{{\rm Soc}(P_i)}$ cover the group ${\rm Soc}(P_i)$). Let $\varphi(x)$ be the formula below expressing in any group $H$ that:
$$\exists y (x\in y^Hy^H \wedge  y^Hy^H\trianglelefteq H  \wedge y^Hy^H\text{ has the $n$-covering property}).$$ Then $\varphi(P_i)$ defines ${\rm Soc}(P_i)$ so ${\rm Soc}(P)=\varphi(P)\equiv \prod_{i\in I} {\rm Soc}(P_i)/\mathcal{U}$ as we wanted.

Finally, instead of `$m_2(P) \leqslant 2$' assume that $P \leqslant G$, where $G$ is as in the claim. Then $m_2(G) \leqslant 2$ by the structure of the Sylow $2$-subgroups in $G$ \cite[Proposition 27]{Deloro-Jaligot} and therefore, by above, $m_2(G)=2$ and ${\rm Soc}(P)$ is isomorphic to one of the groups ${\rm PSL}_2(F), {\rm PSL}_3(F)$ or ${\rm PSU}_3(F)$ over a pseudofinite field $F$ of characteristic $\neq 2$. Combining the fact that $m_2(G)=2$ and \cite[Proposition 27]{Deloro-Jaligot}, we get that the Sylow $2$-subgroups of $G$ are isomorphic to $\mathbb{Z}_{2^\infty} \rtimes \langle w \rangle$ where $w$ is an involution inverting $\mathbb{Z}_{2^\infty}$. Therefore, by \L os's Theorem, the simple group ${\rm Soc}(P_i)$ has dihedral Sylow $2$-subgroups (for detail, see \cite[Proposition 1.4]{Karhumaki-Ugurlu}). It then follows from the Gorenstein-Walter classification of finite simple groups with dihedral Sylow $2$-subgroups \cite{Gorenstein-Walter} (another important part of CFSG which does not use CFSG itself) that ${\rm Soc}(P)\cong {\rm PSL}_2(F)$. \end{proof}

We may now conclude that Proposition~\ref{th:CFSG} holds:

\begin{propo}\label{propo:CFSG}Let $G$ be as in Conjecture~\ref{conj:tight} with ${\rm pr}_2(G)=1$. Then by \cite[Theorem 1.2]{Karhumaki-Ugurlu} $G\cong {\rm PGL}_2(K)$ for $K$ an algebraically closed field of characteristic $\neq 2$. This identification does not depend on CFSG.
\end{propo}
\begin{proof} The identification of $G$ in \cite{Karhumaki-Ugurlu} uses Fact~\ref{propo:pf} which uses CFSG. However, this is the only place in the proof where CFSG is needed. Note also that Fact~\ref{propo:pf}(b) uses part (a) but other than that it does no use CFSG. Now, in the proof of \cite[Theorem 1.2]{Karhumaki-Ugurlu}, instead of using Fact~\ref{propo:pf}(a) one may use Proposition~\ref{propo:pf-noCFSG} (if needed, see \cite[page 10]{Karhumaki-Ugurlu}). This allows one to remove the use of CFGS from the proof.\end{proof}

We finish the paper with a comment on further attempts of removing the use CFSG from the supertight programme. As noted earlier, it seems unlikely that one will reprove Fact~\ref{propo:pf} without CFSG. However, one could hope to prove Conjecture~\ref{conj:tight} without CFSG using the following strategy. If ${\rm pr}_2(G)=2$ then one hopes to identify the socle $S_n={\rm Soc}(C_G(\alpha^n))$ using the classification of finite simple groups with ${\rm BN}$-pair of \emph{Tits rank $2$}. It is not obvious (without CFSG) that ${\rm pr}_2(G)=2$ implies that $S_n$ is elementary equivalent to an ultraproduct of finite simple groups with ${\rm BN}$-pair of Tits rank $2$ but it seems feasible to prove this. In the case ${\rm pr}_2(G) \geqslant 3$, it does not seems feasible to identify the $S_n$'s. Instead, one could hope that using similar techniques as in \cite{Berkman-Borovik2011}, one may identify $G$ without using CFSG (i.e. proving a new version of \cite[Theorem 1.1]{Berkman-Borovik2011} in the supertight set-up, in which we only know that $S_n$'s are elementary equivalent to non-principal ultraproducts of finite simple groups).

\section*{Acknowledgements} I would like to thank P\i nar U\u{g}urlu Kowalski for numerous helpful discussions related to the supertight programme and P\i nar U\u{g}urlu Kowalski and Adrien Deloro for fruitful discussions during the meeting `Groups of finite Morley rank, new directions' at the Nesin Mathematical Village. I would also like to thank Alexandre Borovik for his several useful comments.

\bibliographystyle{plain}
\bibliography{Ulla.2021}

\end{document}